\index{\footnote{}}
\documentclass[11pt]{article}
\usepackage{amssymb,amsmath,amsfonts,amsthm}

\usepackage{epsfig}

\numberwithin{equation}{section}

\newcommand{\di}{\displaystyle}

\newcommand{\R}{{\mathbb R}}
\newcommand{\intr}{\int\limits_{\R^N}}
\newcommand{\ia}{I_\alpha}

\theoremstyle{plain}

\newtheorem{cor}{Corollary}

\newtheorem{theorem}{Theorem}[section]
\newtheorem{lemma}{Lemma}[section]

\theoremstyle{definition}

\begin{document}

\title{Nonlocal Pertubations of Fractional Choquard Equation}

\author{Gurpreet Singh\footnote{School of Mathematics and Statistics,
    University College Dublin, Belfield, Dublin 4, Ireland; {\tt gurpreet.singh@ucdconnect.ie}} 
}

\maketitle

\begin{abstract}
We study the equation
\begin{equation*}
(-\Delta)^{s}u+V(x)u= (I_{\alpha}*|u|^{p})|u|^{p-2}u+\lambda(I_{\beta}*|u|^{q})|u|^{q-2}u \quad\mbox{ in } \R^{N},
\end{equation*}
where $I_\gamma(x)=|x|^{-\gamma}$ for any $\gamma\in (0,N)$, $p, q >0$, $\alpha,\beta\in (0,N)$, $N\geq 3$ and $\lambda \in \R$. First, the existence of a groundstate solutions using minimization method on the associated Nehari manifold is obtained. Next,  the existence of least energy sign-changing solutions is investigated by considering the Nehari nodal set.

\medskip

\noindent {\it Keywords}: Choquard equation; fractional operator; nonlocal perturbation; groundstate solution; sign-changing solutions

\noindent{MSC 2010:} 35R11; 35J60; 35Q40; 35A15

\end{abstract}

\section{Introduction}
In this paper we are concerned with the equation
\begin{equation}\label{fr}
(-\Delta)^{s}u+V(x)u= (I_{\alpha}*|u|^{p})|u|^{p-2}u+\lambda(I_{\beta}*|u|^{q})|u|^{q-2}u \quad\mbox{ in } \R^{N},
\end{equation}
where $p, q >0$, $\alpha,\beta\in (0,N)$, $N\geq 3$. Here  $I_\gamma$ stands for the Riesz potential of order $\gamma$ defined as $I_\gamma=|x|^{\gamma-N}$ for any $\gamma\in (0,N)$. 

The operator $(-\Delta)^s$ is the fractional Laplace operator of order $s\in (0,1)$ and is referred to the infinitesimal generator of Levy stable diffusion process and is defined as follows (see \cite{NPV2012, MRS2016}):
$$
(-\Delta)^s u=C(N,s)\, P.V. \int_{\R^N}\frac{u(x)-u(y)}{|x-y|^{N+2s}}dy,
$$
where $P.V.$ stands for the principal value of the integral and $C(N,s)>0$ is a normalizing constant.
The function $V \in C(\R^{N})$ is required to satisfy one (or both) of the following conditions:

\begin{enumerate}
\item[(V1)] $\inf_{\R^{N}}V(x)\geq V_{0}> 0$ ;
\item[(V2)] For all $M>0$ the set $\{x\in \R^N: V(x)\leq M\}$ has finite Lebesgue measure. 
\end{enumerate}

Note that condition $(V2)$ is weaker than $\lim_{|x|\to \infty}V(x)=\infty$ as for instance $V(x)=|x|^4 \sin^2 |x|$ satisfies $(V2)$ but has no limit as $|x|\to \infty$.
 
In the last few decades, problems involving fractional Laplacian and non-local operators have received considerable attention. These kind of problems arise in various applications, such as continuum mechanics, phase transitions, population dynamics, optimization, finance and many others. 

The prototype model of \eqref{fr} is the fractional Choquard equation 
\begin{equation}\label{squ}
(-\Delta)^{s}u+V(x)u= (I_{\alpha}*|u|^{p})|u|^{p-2}u \quad\mbox{ in } \R^{N},
\end{equation}
studied by Avenia, Siciliano and Squassina in \cite{DSS2015} in case where $V$ is a positive constant. The authors in \cite{DSS2015} obtained the existence of  groundstate and radially symmetric solutions with diverging norm and diverging energy levels.

The case of standard Laplace operator in \eqref{squ} has a long history in the literature.
For $s=1$, $V\equiv 1$, $p=\alpha= 2$, the equation \eqref{squ} becomes the well known Choquard or nonlinear Schr\"odinger-Newton equation 
\begin{equation}\label{fr1a}
-\Delta u+ u= (I_{2}*u^{2})u \quad\mbox{ in } \R^{N}.
\end{equation}
The equation \eqref{fr1a} for $N=3$ was first introduced by S.I. Pekar in 1954 in quantum mechanics. In 1996, R. Penrose \cite{P1996, P1998} used equation \eqref{fr1a} in a different context as a model in self-gravitating matter (see  also \cite{J1995, MPT1998}). Since then, the  Choquard equation has been investigated in various settings and in many contexts (see, e.g., \cite{AFY2016, GT2016, MZ2010}).  For a most up to date reference on the study of Choquard equation in a standard Laplace setting the reader may consult \cite{MV2017}.  

For $s= 1/2$, $V\equiv 1$, $p=\alpha= 2$, $N=3$ and $\lambda= 0$, the equation \eqref{fr} becomes
\begin{equation}\label{fr1b}
(-\Delta)^{\frac{1}{2}} u+ u= (I_{2}*u^{2})u \quad\mbox{ in } \R^{3},
\end{equation}
and has been used to study the dynamics of pseudo-relativistic boson stars and their dynamical evolution (see \cite{ES2007, FJL2007, FJL2007a, L2009}).

In this paper we shall be interested in the study of groundstate solutions and least energy sign-changing solutions to \eqref{fr}. To this aim, we denote by ${\mathcal D}^{2,s}(\R^{N})$ the completion of $C^{\infty}_c(\R^{N})$ with respect to the Gagliardo seminorm
$$
[u]_{s,2}= \Big[\int_{{\mathbb R}^N}\!\!\!\int_{{\mathbb R}^N}\frac{|u(x)-u(y)|^{2}}{|x-y|^{N+2s}}dxdy\Big]^{\frac{1}{2}}.
$$
Also, $H^{s}(\R^{N})$ denotes the standard fractional Sobolev space defined as the set of $u \in {\mathcal D}^{2,s}(\R^{N})$ satisfying $u \in L^{2}(\R^{N})$ with the norm
$$
\|u\|_{H^{s}}= \Big[ \int_{\R^{N}}|u|^{2}+ [u]_{s,2}^{2}\Big]^{\frac{1}{2}}.
$$
Let us define the functional space
$$
X_{V}^{s}(\R^{N})= \Big\{u\in {\mathcal D}^{2,s}(\R^{N}): \int_{\R^{N}}V(x)u^2< \infty \Big\},
$$
endowed with the norm
$$
\|u\|_{X_{V}^{s}}=\Big[\int_{{\mathbb R}^N}\!\!\!\int_{{\mathbb R}^N}\frac{|u(x)-u(y)|^{2}}{|x-y|^{N+2s}}dxdy+\int_{\R^{N}}V(x)u^2 \Big]^{\frac{1}{2}}.
$$

Throughout this paper we shall assume that $p$ and $q$ satisfy
\begin{equation}\label{p}
\frac{N+\alpha}{N}< p< \frac{N+\alpha}{N-2s},
\end{equation}
and
\begin{equation}\label{q}
\frac{N+\beta}{N}< q< \frac{N+\beta}{N-2s}.
\end{equation}
It not difficult to see that \eqref{fr} has a variational structure. Indeed,  any solution of \eqref{fr} is a critical point of the energy functional 
${\mathcal E}_\lambda:X_{V}^{s}(\R^{N}) \rightarrow \R$ defined by
\begin{equation}\label{fr1}
\begin{aligned}
{\mathcal E}_\lambda(u)&=\frac{1}{2}\|u\|_{X_{V}^{s}}^{2}-\frac{1}{2p}\int_{\R^{N}}(I_{\alpha}*|u|^{p})|u|^{p}-\frac{\lambda}{2q}\int_{\R^{N}}(I_{\beta}*|u|^{q})|u|^{q}.
\end{aligned}
\end{equation}

A crucial tool to our approach is the following Hardy-Littlewood-Sobolev inequality
\begin{equation}\label{hli}
\Big| \int_{\R^N}(I_\gamma*u)v \Big| \leq C\|u\|_r \|v\|_t,	
\end{equation}
for $\gamma \in (0, N)$ and $u\in L^{r}(\R^N)$, $v\in L^{t}(\R^N)$ such that
$$
\frac{1}{r}+\frac{1}{t}=1+\frac{\gamma}{N}.
$$
Using \eqref{p}, \eqref{q} together with the Hardy-Littlewood-Sobolev inequality \eqref{hli}, the energy functional ${\mathcal E}_\lambda$ is well defined and  moreover ${\mathcal E}_\lambda\in C^1(X_V^s)$.

We shall first be concerned with the existence of ground state solutions for the equation \eqref{fr} under the assumption that $V$ satisfies $(V1)$. This will be achieved by minimisation method on the Nehari manifold associated with ${\mathcal E}_\lambda$, which is defined as 
\begin{equation}\label{nm}
{\mathcal N_\lambda}=\{u\in X_{V}^{s}(\R^N)\setminus\{0\}: \langle {\mathcal E}_\lambda'(u),u\rangle=0\}.
\end{equation}

The groundstate solutions will be obtained as minimizers of
$$
m_{\lambda}=\inf_{u\in {\mathcal N_\lambda}}{\mathcal E}_\lambda(u).
$$
Our main result in this sense is stated below.
\begin{theorem}\label{gstate}
Assume $p>q> 1$,  $\lambda> 0$, $p,q$ satisfy \eqref{p}-\eqref{q} and $V$ satisfies $(V1)$. Then the equation \eqref{fr} has a ground state solution $u\in X_{V}^{s}(\R^{N})$. 
\end{theorem}
Our approach relies on the analysis of the Palais-Smale sequences for ${\mathcal E}_\lambda \!\mid_{\mathcal N_\lambda}$. Using an idea from \cite{CM2016,CV2010} we show that any Palais-Smale sequence of ${\mathcal E}_\lambda \!\mid_{\mathcal N_\lambda}$ is either converging strongly to its weak limit or differs from it by a finite number of sequences which further are the translated solutions of \eqref{squ}. The novelty of our approach is that we shall rely on several nonlocal Brezis-Lieb results as we present in Section $3.2$.

\medskip

We now turn to the study of least energy sign-changing solutions of \eqref{fr}. In this setting we require $V$ to fulfill both $(V1)$ and $(V2)$. By the result in \cite[Lemma 2.1]{T2017} (see also \cite{PXZ2015,PXZ2016}) the embedding $X_{V}^{s}(\R^{N}) \hookrightarrow L^{q}(\R^{N})$ is compact for $q \in [2, 2_{s}^{*})$ where $2_{s}^{*}= \frac{2N}{N-2s}$. 

Our approach in the study of least energy sign-changing solutions of \eqref{fr} is based 
on the minimisation method on the Nehari nodal set defined as 
$$
{\cal M_\lambda}= \Big\{u\in  X_{V}^{s}(\R^{N}): u^{\pm} \neq 0 \mbox{ and } \langle {\mathcal E}_\lambda'(u),u^{\pm}\rangle  \mbox{ = } 0 \Big\}.
$$
The solutions will be obtained as a minimizers for 
$$
c_\lambda= \inf_{u\in {\cal M}_{\lambda}}{\mathcal E}_\lambda(u).
$$
In this situation the problem is more delicate as some of the usual properties of the local nonlinear functional doesn't work. For instance, since
\begin{equation*}
\begin{aligned}
\langle {\mathcal E}_\lambda'(u),u^{\pm} \rangle = &\|u^{\pm}\|_{X_{V}^{s}}^{2}-\int_{\R^{N}}(I_{\alpha}*(u^{\pm})^{p})(u^{\pm})^{p}-\lambda\int_{\R^{N}}(I_{\beta}*(u^{\pm})^{q})(u^{\pm})^{q}\\
&- \int_{{\mathbb R}^N}\!\!\!\int_{{\mathbb R}^N}\frac{u^{\pm}(x)u^{\mp}(y)+u^{\mp}(x)u^{\pm}(y)}{|x-y|^{N+2s}}dxdy
- \int_{\R^{N}}(I_{\alpha}*(u^{\pm})^{p})(u^{\mp})^{p}\\
&-\lambda\int_{\R^{N}}(I_{\beta}*(u^{\pm})^{q})(u^{\mp})^{q},
\end{aligned}
\end{equation*}
we have in general that  
\begin{equation*}
{\mathcal E}_\lambda(u) \neq 	{\mathcal E}_\lambda(u^+)+{\mathcal E}_\lambda(u^-) \quad\mbox{ and } \quad\langle {\mathcal E}^{'}_{\lambda}(u), u^{\pm} \rangle \neq  \langle {\mathcal E}^{'}_{\lambda}(u^{\pm}), u^{\pm} \rangle.
\end{equation*}
Therefore, the standard local methods used to investigate the existence of sign-changing solutions do not apply immediately to our nonlocal setting. 

Our second main result in this regard is stated below.
\begin{theorem}\label{signc}
Assume $\lambda \in \R$, $(N-4s)_{+}< \alpha, \beta< N$, $p>q> 2$ satisfy \eqref{p} and \eqref{q} and $V$ satisfies $(V1)$ and $(V2)$. Then the equation \eqref{fr} has a least-energy sign-changing solution $u\in X_{V}^{s}(\R^{N})$.
\end{theorem}

\medskip

The remaining of the paper is organized as follows. In Section $2$ we collect  some nonlocal versions of the Bezis-Lieb lemma which will be crucial in investigating the groundstate solutions of \eqref{fr}. Further,  Section $3$ and $4$ contain the proofs of our main results.

\bigskip

\section{Preliminary results}

\begin{lemma}\label{cc}(\cite[Lemma 1.1]{L1984}, \cite[Lemma 2.3]{MV2013})

Let $r\in [2,2_{s}^*]$. There exists a constant $C>0$ such that for any $u\in X_{V}^{s}(\R^{N})$ we have
$$
\int_{\R^N}|u|^r \leq C||u||\Big(\sup_{y\in \R^N} \int_{B_1(y)}|u|^r \Big)^{1-\frac{2}{r}}.
$$
\end{lemma}

\begin{lemma}\label{bogachev}(\cite[Proposition 4.7.12]{B2007})

Let $r\in (1,\infty)$. Assume $(w_n)$ is a bounded sequence in $L^r(\R^N)$ that converges to $w$ almost everywhere. Then $w_n\rightharpoonup w$ weakly in $L^r(\R^N)$.
\end{lemma}

\begin{lemma}\label{blc}{\sf (Local Brezis-Lieb lemma)}

Let $r\in (1,\infty)$. Assume $(w_n)$ is a bounded sequence in $L^r(\R^N)$ that converges to $w$ almost everywhere. Then, for every $q\in [1,r]$ we have
$$
\lim_{n\to\infty}\int_{\R^N}\big| |w_n|^q-|w_n-w|^q-|w|^q\big|^{\frac{r}{q}}=0\,,
$$
and
$$
\lim_{n\to\infty}\int_{\R^N}\big| |w_n|^{q-1}w_n-|w_n-w|^{q-1}(w_n-w)-|w|^{q-1}w\big|^{\frac{r}{q}}=0.
$$
\end{lemma}
\begin{proof}
Fix $\varepsilon>0$. Then there exists $C(\varepsilon)>0$ such that for all $a$,$b\in \R$ we have
\begin{equation}\label{lb1}
\Big||a+b|^{q}-|a|^{q}\big|^{\frac{r}{q}}\leq \varepsilon|a|^{r}+C(\varepsilon)|b|^{r}.
\end{equation}
Using \eqref{lb1}, we obtain
$$
\begin{aligned}
|f_{n, \varepsilon}|=& \Big( \Big||w_{n}|^{q}-|w_{n}-w|^{q}-|w^{q}|\Big|^{\frac{r}{q}}-\varepsilon|w_{n}-w|^{r}\Big)^{+}\\
&\leq (1+C(\varepsilon))|w|^{r}.
\end{aligned}
$$
Now using Lebesgue Dominated Convergence theorem, we have
\begin{equation}
\intr {f_{n, \varepsilon}} \rightarrow 0  \quad\mbox{ as } n\rightarrow \infty.
\end{equation}

Therefore, we get
$$
\Big||w_{n}|^{q}-|w_{n}-w|^{q}-|w|^{q}\Big|^{\frac{r}{q}}\leq f_{n, \varepsilon}+\varepsilon|w_{n}-w|^{r},
$$
which gives
$$
\limsup_{n\rightarrow \infty} {\intr \Big||w_{n}|^{q}-|w_{n}-w|^{q}-|w|^{q}\Big|^{\frac{r}{q}}}\leq c\varepsilon,
$$
where $c= \sup_{n}|w_{n}-w|_{r}^{r}< \infty$. Further letting $\varepsilon \rightarrow 0$, we conclude the proof.
\end{proof}

\begin{lemma}\label{nlocbl}{\sf (Nonlocal Brezis-Lieb lemma}(\cite[Lemma 2.4]{MV2013})

Let $\alpha\in (0,N)$ and $p\in [1,\frac{2N}{N+\alpha})$. Assume $(u_n)$ is a bounded sequence in $L^{\frac{2Np}{N+\alpha}}(\R^N)$ that converges almost everywhere to  some $u:\R^N\to \R$. Then
$$
\lim_{n\to \infty}\int_{\R^N}\Big|(I_\alpha*|u_n|^p)|u_n|^p-(I_\alpha*|u_n-u|^p)|u_n-u|^p-(I_\alpha*|u|^p)|u|^p  \Big|=0.
$$
\end{lemma}
\begin{proof} For $n\in N$, we observe that
\begin{equation}\label{nb1}
\begin{aligned}
\intr \Big[(\ia*|u_n|^p)|u_n|^{p}-(\ia*(|u_n-u|^p))(|u_n-u|^{p})\Big]&=\intr [\ia*(|u_n|^p-|u_n-u|^p)](|u_n|^p-|u_n-u|^p)\\
&+2\intr [\ia*(|u_n|^p-|u_n-u|^p)|u_n-u|^p.
\end{aligned}
\end{equation}
Using Lemma \ref{blc} with $q=p$, $r=\frac{2Np}{N+\alpha}$, we have $|u_n-u|^p-|u_n|^p\to |u|^p$ strongly in $L^{\frac{2N}{N+\alpha}}(\R^N)$ and by Lemma \ref{bogachev} we get $|u_n-u|^{p}\rightharpoonup 0$ weakly in $L^{\frac{2N}{N+\alpha}}(\R^{N})$. Also by Hardy-Littlewood-Sobolev inequality \eqref{hli} we obtain 
$$
\ia*(|u_n-u|^p-|u_n|^{p})\to \ia*|u|^p \quad\mbox{ in } L^{\frac{2N}{N-\alpha}}(\R^N).
$$
Using all the above arguments and passing to the limit in \eqref{nb1} we conclude the proof.

\end{proof}

\begin{lemma}\label{anbl}

Let $\alpha\in (0,N)$ and $p\in [1,\frac{2N}{N+\alpha})$. Assume $(u_n)$ is a bounded sequence in $L^{\frac{2Np}{N+\alpha}}(\R^N)$ that converges almost everywhere to $u$. Then, for any $h\in L^{\frac{2Np}{N+\alpha}}(\R^N)$ we have
$$
\lim_{n\to \infty}\int_{\R^N}(I_\alpha*|u_n|^p)|u_n|^{p-2}u_nh=\int_{\R^N} (I_\alpha*|u|^p)|u|^{p-2}uh.
$$
\end{lemma}
\begin{proof} Using $h=h^+-h^-$, it is enough to prove our lemma for $h\geq 0$. 
Denote $v_n=u_n-u$ and observe that
\begin{equation}\label{bl1}
\begin{aligned}
\intr (\ia*|u_n|^p)|u_n|^{p-2}u_nh=& \intr [\ia*(|u_n|^p-|v_n|^p)](|u_n|^{p-2}u_nh-|v_n|^{p-2}v_nh)\\
&+\intr [\ia*(|u_n|^p-|v_n|^p)]|v_n|^{p-2}v_nh\\
&+\intr [\ia*(|u_n|^{p-2}u_nh-|v_n|^{p-2}v_n h)|v_n|^p\\
&+\intr (\ia*|v_n|^p)|v_n|^{p-2}v_nh.
\end{aligned}
\end{equation}
Apply Lemma \ref{blc} with $q=p$, $r=\frac{2Np}{N+\alpha}$ by taking respectively $(w_n,w)=(u_n,u)$ and then $(w_n,w)=(u_nh^{1/p}, u h^{1/p})$. We find
$$
\left\{
\begin{aligned}
&|u_n|^p-|v_n|^p\to |u|^p \\
&|u_n|^{p-2}u_nh-|v_n|^{p-2}v_nh\to |u|^{p-2}uh
\end{aligned}
\right.
\quad\mbox{ strongly in }\; L^{\frac{2N}{N+\alpha}}(\R^N).
$$
Using now the Hardy-Littlewood-Sobolev inequality we obtain
\begin{equation}\label{est00}
\left\{
\begin{aligned}
&\ia*(|u_n|^p-|v_n|^p)\to \ia*|u|^p \\
&\ia*(|u_n|^{p-2}u_nh-|v_n|^{p-2}v_nh)\to \ia*(|u|^{p-2}uh)
\end{aligned}
\right.
\quad\mbox{ strongly in }\; L^{\frac{2N}{N-\alpha}}(\R^N).
\end{equation}
Also, by Lemma \ref{bogachev} we have
\begin{equation}\label{est01}
|u_n|^{p-2}u_n h\rightharpoonup |u|^{p-2}uh,\; |v_n|^p\rightharpoonup 0,\; |v_n|^{p-2}v_nh\rightharpoonup 0\quad \mbox{ weakly in }\; L^{\frac{2N}{N+\alpha}}(\R^N).
\end{equation}
Combining \eqref{est00}-\eqref{est01} we find
\begin{equation}\label{est02}
\left\{
\begin{aligned}
&\lim_{n\to \infty} \intr [\ia*(|u_n|^p-|v_n|^p)](|u_n|^{p-2}u_nh-|v_n|^{p-2}v_n h)=\int_{\R^N} (I_\alpha*|u|^p)|u|^{p-2}uh,\\
&\lim_{n\to \infty} \intr [\ia*(|u_n|^p-|v_n|^p)]|v_n|^{p-2}v_nh=0,\\
&\lim_{n\to \infty} \intr [\ia*(|u_n|^{p-2}u_nh-|v_n|^{p-2}v_nh)|v_n|^p=0.
\end{aligned}
\right.
\end{equation}
By H\"older's inequality and Hardy-Littlewood-Sobolev with we have
\begin{equation}\label{est03}
\begin{aligned}
\left| \intr (\ia*|v_n|^{p}))|v_n|^{p-2}v_nh \right|
&\leq \|v_n\|^p_{\frac{2Np}{N+\alpha}}\||v_n|^{p-1}h\|_{\frac{2N}{N+\alpha}}\\
&\leq C \||v_n|^{p-1}h\|_{\frac{2N}{N+\alpha}}.
\end{aligned}
\end{equation}
On the other hand, by Lemma \ref{bogachev} we have $v_n^{\frac{2N(p-1)}{N+\alpha}}\rightharpoonup 0$ weakly in 
$L^{\frac{p}{p-1}}(\R^N)$ so 
$$
\||v_n|^{p-1}h\|_{\frac{2N}{N+\alpha}}=\left(\intr |v_n|^{\frac{2N(p-1)}{N+\alpha}}|h|^{\frac{2N}{N+\alpha}}  \right)^{\frac{N+\alpha}{2N}}\to 0.
$$
Thus, from \eqref{est03}  have
\begin{equation}\label{est04}
\lim_{n\to \infty} \intr (\ia*|v_n|^{p}))|v_n|^{p-2}v_nh=0.
\end{equation}
Passing to the limit in \eqref{bl1}, from \eqref{est02} and \eqref{est04} we reach the conclusion.
\end{proof}

\bigskip

\section{Proof of Theorem \ref{gstate}}
In this section, we discuss the existence of a groundstate solutions to \eqref{fr} under the assumption $\lambda> 0$.  For $u,v\in X_{V}^{s}(\R^{N})$  we have
\begin{equation*}
\begin{aligned}
\langle {\mathcal E}_\lambda'(u),v \rangle &= \int_{{\mathbb R}^N}\!\!\!\int_{{\mathbb R}^N}\frac{(u(x)-u(y))(v(x)-v(y))}{|x-y|^{N+2s}}dxdy+\int_{\R^{N}}V(x)uv \\
&-\int_{\R^{N}}(I_{\alpha}*|u|^{p})|u|^{p-1}v -\lambda\int_{\R^{N}}(I_{\beta}*|u|^{q})|u|^{q-1}v.
\end{aligned}
\end{equation*}
So, for $t>0$ we have
$$
\langle {\mathcal E}_\lambda'(tu),tu \rangle= t^{2}\|u\|_{X_{V}^{s}}^{2}-t^{2p}\int_{\R^{N}}(I_{\alpha}*|u|^{p})|u|^{p}-\lambda t^{2q}\int_{\R^{N}}(I_{\beta}*|u|^{q})|u|^{q}
$$
Since $p>q>1$, the equation
$$
\langle {\mathcal E}_\lambda'(tu),tu \rangle= 0,
$$
has a unique positive solution $t=t(u)$ and the corresponding element $tu\in {\mathcal N_\lambda}$ is called the {\it projection of $u$} on ${\mathcal N_\lambda}$. The next result presents the main properties of the Nehari manifold ${\mathcal N_\lambda}$ which we use in this paper.

\begin{lemma}\label{nehari}

\begin{enumerate}
\item[(i)] ${\mathcal E}_\lambda \!\mid_{\mathcal N_\lambda}$ is bounded from below by a positive constant;
\item[(ii)] Any critical point $u$ of ${\mathcal E}_\lambda \!\mid_{\mathcal N_\lambda}$ is a free critical point.
\end{enumerate}
\end{lemma}
\begin{proof} (i) Using the continuous embeddings  $X_{V}^{s}(\R^N) \hookrightarrow L^{\frac{2Np}{N+\alpha}}(\R^N)$ and  $X_{V}^{s}(\R^N) \hookrightarrow L^{\frac{2Nq}{N+\beta}}(\R^N)$ together with Hardy-Littlewood-Sobolev inequality, for any $u\in {\mathcal N_\lambda}$ we have
$$
\begin{aligned}
0=\langle {\mathcal E}_\lambda'(u),u\rangle& =\|u\|_{X_{V}^{s}}^2-\int_{\R^{N}}(I_{\alpha}*|u|^{p})|u|^{p}-\lambda \int_{\R^{N}}(I_{\beta}*|u|^{q})|u|^{q}\\
&\geq \|u\|_{X_{V}^{s}}^2-C\|u\|_{X_{V}^{s}}^{2p}-C_\lambda \|u\|_{X_{V}^{s}}^{2q}.
\end{aligned}
$$
Therefore, there exists $C_0>0$ such that
\begin{equation}\label{cnot}
\|u\|_{X_{V}^{s}}\geq C_0>0\quad\mbox{for all }u\in {\mathcal N_\lambda}.
\end{equation}
Using the above fact we have
$$
\begin{aligned}
{\mathcal E}_\lambda(u)&= {\mathcal E}_\lambda(u)-\frac{1}{2q}\langle {\mathcal E}_\lambda'(u), u \rangle \\
&=\Big(\frac{1}{2}-\frac{1}{2q}\Big)\|u\|_{X_{V}^{s}}^2+\Big(\frac{1}{2q}-\frac{1}{2p}\Big)\int_{\R^{N}} (I_{\alpha}*|u|^{p})|u|^{p} \\
&\geq \Big(\frac{1}{2}-\frac{1}{2q}\Big)\|u\|_{X_{V}^{s}}^2 \\
&\geq \Big(\frac{1}{2}-\frac{1}{2q}\Big) C_0^2>0.
\end{aligned}
$$

(ii) Let $ {\mathcal L}(u)=\langle {\mathcal E}_\lambda'(u),u\rangle $ for $u \in X_{V}^{s}(\R^N)$. Now, for  $u\in {\mathcal N_\lambda}$, from \eqref{cnot} we get
\begin{equation}\label{cnot1}
\begin{aligned}
\langle {\mathcal L}'(u),u\rangle&=2\|u\|^2-2p\int_{\R^{N}}(I_{\alpha}*|u|^{p})|u|^{p}-2q \lambda\int_{\R^{N}}(I_{\beta}*|u|^{q})|u|^{q}\\
&=2(1-q)\|u\|_{X_{V}^{s}}^2-2(p-q)\int_{\R^{N}}(I_{\alpha}*|u|^{p})|u|^{p}\\
&\leq -2(q-1)\|u\|_{X_{V}^{s}}^2\\
&<-2(q-1)C_0.
\end{aligned}
\end{equation}
Assuming that $u\in {\mathcal N_\lambda}$ is a critical point of ${\mathcal E}_\lambda \!\mid_{\mathcal N_\lambda}$ and using Lagrange multiplier theorem, there exists $\mu \in \R$ such that
${\mathcal E}_\lambda'(u)=\mu {\mathcal L}'(u)$. 
In particular $\langle {\mathcal E}_\lambda '(u),u\rangle=\mu \langle {\mathcal L}'(u),u\rangle$. As $\langle {\mathcal L}'(u),u\rangle<0$, this implies $\mu=0$ so ${\mathcal E}_\lambda '(u)=0$.

\end{proof}

\subsection{Compactness}\label{compc}
Define
$$
{\mathcal J}:X_{V}^{s}(\R^N)\to \R,\quad {\mathcal J}(u)=\frac{1}{2}\|u\|^{2}-\frac{1}{2p}\intr (I_\alpha*|u|^p)|u|^p.
$$
For all $\phi \in C^{\infty}_{0}(\R^N)$, we have
$$
\langle {\mathcal J}'(u), \phi \rangle= \int_{{\mathbb R}^N}\!\!\!\int_{{\mathbb R}^N}\frac{(u(x)-u(y))(\phi(x)-\phi(y))}{|x-y|^{N+2s}}dxdy+\int_{\R^{N}}V(x)u\phi \\
-\int_{\R^{N}}(I_{\alpha}*|u|^{p})|u|^{p-1}\phi,
$$
and
\begin{equation*}
\langle {\mathcal J}'(u),u\rangle= \|u\|_{X_{V}^{s}}^{2}-\int_{\R^{N}}(I_{\alpha}*|u|^{p})|u|^{p}.
\end{equation*}

Also, consider the Nehari manifold associated with ${\mathcal J}$ as
$$
{\mathcal N}_{{\mathcal J}}=\{u\in X_{V}^{s}(\R^N)\setminus\{0\}: \langle {\mathcal J}'(u),u\rangle=0\},
$$
and let 
$$
m_{\mathcal J}=\inf_{u\in {\mathcal N}_{\mathcal J}}{\mathcal J}(u).
$$

\begin{lemma}\label{compact}
Let $(u_n)\subset{\mathcal N}_{\mathcal J}$ be a $(PS)$ sequence of ${\mathcal E}_\lambda \!\mid_{{\mathcal N}_{\lambda}}$, that is,
\begin{enumerate}
\item[(a)] $({\mathcal E}_\lambda(u_n))$ is bounded;
\item[(b)] ${\mathcal E}_\lambda'\!\mid_{{\mathcal N}_{\lambda}}(u_n)\to 0$ strongly in $X_{V}^{-s}(\R^N)$.
\end{enumerate}
Then, there exists a solution $u\in X_{V}^{s}(\R^N)$ of \eqref{fr} such that replacing $(u_n)$ with a subsequence then one of the following alternative holds

\smallskip

\noindent (A) either $u_n\to u$ strongly in $X_{V}^{s}(\R^N)$;

or
\smallskip

\noindent (B) $u_n\rightharpoonup u$ weakly in $X_{V}^{s}(\R^N)$ and there exists 
a positive integer $k\geq 1$ and $k$ functions $u_1,u_2,\dots, u_k\in X_{V}^{s}(\R^N)$ which are nontrivial  weak solutions to \eqref{squ} and $k$ sequences of points $(z_{n,1})$, $(z_{n,2})$, $\dots$, $(z_{n,k})\subset \R^N$
such that:
\begin{enumerate}
\item[(i)] $|z_{n,j}|\to \infty$ and $|z_{n,j}-z_{n,i}|\to \infty$  if $i\neq j$, $n\to \infty$;
\item[(ii)] $\di u_n-\sum_{j=1}^ku_j(\cdot+z_{n,j})\to u$ in $X_{V}^{s}(\R^N)$;
\item[(iii)] $\di {\mathcal E}_\lambda(u_n)\to {\mathcal E}_{\lambda}(u)+\sum_{j=1}^k {\mathcal J}(u_j)$.
\end{enumerate}
\end{lemma}
\begin{proof}
Since $(u_n)$ is bounded in $X_{V}^{s}(\R^N)$, there exists $u\in X_{V}^{s}(\R^N)$ such that, up to a subsequence, we have
\begin{equation}\label{firstconv}
\left\{
\begin{aligned}
u_n& \rightharpoonup u \quad\mbox{ weakly in }X_{V}^{s}(\R^N),\\
u_n &\rightharpoonup u\quad\mbox{ weakly in }L^r(\R^N),\; 2\leq r\leq 2_{s}^*,\\
u_n & \to u\quad\mbox{ a.e. in }\R^N.
\end{aligned}
\right.
\end{equation}

Using \eqref{firstconv} and Lemma \ref{anbl} it follows that ${\mathcal E}_\lambda'(u)=0$, so $u\in X_{V}^{s}(\R^N)$ is a solution of \eqref{fr}. Further, if $u_n\to u$ strongly in $X_{V}^{s}(\R^N)$ then $(A)$ in Lemma \ref{compact} holds.

\medskip

Now, assume that $(u_n)$ does not converge strongly to $u$ in $X_{V}^{s}(\R^N)$ and set
$w_{n,1}=u_n-u$. Then $(w_{n,1})$ converges weakly to zero in $X_{V}^{s}(\R^N)$ and
\begin{equation}\label{bl2}
\|u_n\|_{X_{V}^{s}}^2=\|u\|_{X_{V}^{s}}^2+\|w_{n,1}\|_{X_{V}^{s}}^2+o(1).
\end{equation}
By Lemma \ref{nlocbl} we have
\begin{equation}\label{bl3}
\int_{\R^N} (I_\alpha*|u_n|^p)|u_n|^p=\intr (I_\alpha*|u|^p)|u|^p+\intr (I_\alpha*|w_{n,1}|^p)|w_{n,1}|^p+o(1). 
\end{equation}
 Using \eqref{bl2} and \eqref{bl3} we get
\begin{equation}\label{est6}
{\mathcal E}_\lambda(u_n)= {\mathcal E}_\lambda(u)+{\mathcal J}(w_{n,1})+o(1).
\end{equation}
Further, for any $h\in X_{V}^{s}(\R^{N})$, by Lemma \ref{anbl}  we have
\begin{equation}\label{est7}
\langle{\mathcal J}'(w_{n,1}), h\rangle=o(1).
\end{equation}
From Lemma \ref{nlocbl} we deduce that
$$
\begin{aligned}
0=\langle {\mathcal E}_\lambda'(u_n), u_n \rangle&=\langle {\mathcal E}_\lambda'(u),u\rangle+\langle {\mathcal J}'(w_{n,1}), w_{n,1} \rangle+o(1)\\
&=\langle{\mathcal J}'(w_{n,1}), w_{n,1}\rangle+o(1).
\end{aligned}
$$
This implies
\begin{equation}\label{est8}
\langle {\mathcal J}'(w_{n,1}), w_{n,1}\rangle=o(1).
\end{equation}
We need the following auxillary result:
\begin{lemma}\label{c1}
Define
$$
\delta:=\limsup_{n\to \infty}\Big(\sup_{w\in \R^N} \int_{B_1(z)}|w_{n,1}|^{\frac{2Np}{N+\alpha}}\Big).
$$
Then $\delta>0$.
\end{lemma}
\begin{proof}
Assume by contradiction $\delta= 0$. By Lemma \ref{cc} we deduce that $w_{n,1}\to 0$ strongly in $L^{\frac{2Np}{N+\alpha}}(\R^N)$. Then, by Hardy-Littlewood-Sobolev inequality we get
$$
\int_{\R^N} (I_\alpha*|w_{n,1}|^p)|w_{n,1}|^p=o(1).
$$
Using this fact together with \eqref{est8}, we get $w_{n,1}\to 0$ strongly in $X_{V}^{s}(\R^{N})$. This is a contradiction. Hence, $\delta> 0$.
\end{proof}
Now, we return to the proof of Lemma \ref{compact}. Since $\delta>0$, we may find $z_{n,1}\in \R^N$ such that
\begin{equation}\label{est9}
\int_{B_1(z_{n,1})}|w_{n,1}|^{\frac{2Np}{N+\alpha}}>\frac{\delta}{2}.
\end{equation}
Consider the sequence $(w_{n,1}(\cdot+z_{n,1}))$. Then there exists $u_1\in X_{V}^{s}(\R^{N})$ such that, up to a subsequence, we have  
$$
\begin{aligned}
w_{n,1}(\cdot+z_{n,1})&\rightharpoonup u_1\quad\mbox{ weakly in } X_{V}^{s}(\R^{N}),\\
w_{n,1}(\cdot+z_{n,1})&\to u_1\quad\mbox{ strongly in } L_{loc}^{\frac{2Np}{N+\alpha}}(\R^N),\\
w_{n,1}(\cdot+z_{n,1})&\to u_1\quad\mbox{ a.e. in } \R^N.
\end{aligned}
$$
Next, passing to the limit in \eqref{est9} we get
$$
\int_{B_1(0)}|u_{1}|^{\frac{2Np}{N+\alpha}}\geq \frac{\delta}{2},
$$
therefore, $u_1\not\equiv 0$. Since $(w_{n,1})$ converges weakly to zero in $X_{V}^{s}(\R^{N})$  it follows that $(z_{n,1})$ is unbounded and then passing to a subsequence we may assume that $|z_{n,1}|\to \infty$. By \eqref{est8} we deduce that ${\mathcal J}'(u_1)=0$, so $u_1$ is a nontrivial solution of \eqref{squ}.

Further, define
$$
w_{n,2}(x)=w_{n,1}(x)-u_1(x-z_{n,1}).
$$
Similarly as before we have
$$
\|w_{n,1}\|^2=\|u_1\|^2+\|w_{n,2}\|^2+o(1).
$$
and then using Lemma \ref{nlocbl} we deduce that
$$
\int_{\R^N} (I_\alpha*|w_{n,1}|^p)|w_{n,1}|^p=\int_{\R^N} (I_\alpha*|u_1|^p)|u_1|^p+\intr (I_\alpha*|w_{n,2}|^p)|w_{n,2}|^p+o(1). 
$$
Hence,
$$
{\mathcal J}(w_{n,1})={\mathcal J}(u_1)+{\mathcal J}(w_{n,2})+o(1).
$$
So, by \eqref{est6} one can get
$$
{\mathcal E}_\lambda (u_n)= {\mathcal E}_\lambda (u)+{\mathcal J}(u_1)+{\mathcal J}(w_{n,2})+o(1).
$$
Using the above techniques, we also obtain
$$
\langle {\mathcal J}'(w_{n,2}),h\rangle =o(1)\quad\mbox{ for any }h\in X_{V}^{s}(\R^{N})
$$
and
$$
\langle {\mathcal J}'(w_{n,2}),w_{n,2}\rangle =o(1).
$$
Now, if $(w_{n,2})$ converges strongly to zero, then we finish the proof by taking $k=1$ in the statement of Lemma \ref{compact}. If $w_{n,2}\rightharpoonup 0$ weakly and not strongly in $X_{V}^{s}(\R^{N})$, then we iterate the process. In $k$ number of steps one could find a set of sequences $(z_{n,j})\subset \R^N$, $1\leq j\leq k$ with 
$$
|z_{n,j}|\to \infty\quad\mbox{  and }\quad |z_{n,i}-z_{n,j}|\to \infty\quad\mbox{  as }\; n\to \infty, i\neq j
$$
and $k$ nontrivial solutions  $u_1$, $u_2$, $\dots$, $u_k\in X_{V}^{s}(\R^{N})$ of \eqref{squ} such that, denoting 
$$
w_{n,j}(x):=w_{n,j-1}(x)-u_{j-1}(x-z_{n,j-1})\,, \quad 2\leq j\leq k,
$$ 
we have
$$
w_{n,j}(x+z_{n,j})\rightharpoonup u_j\quad\mbox{weakly in }\; X_{V}^{s}(\R^{N})
$$
and
$$
{\mathcal E}_\lambda(u_n)= {\mathcal E}_\lambda(u)+\sum_{j=1}^k {\mathcal J}(u_j)+{\mathcal J}(w_{n,k})+o(1).
$$
As ${\mathcal E}_\lambda (u_n)$ is bounded and ${\mathcal J}(u_j)\geq m_{\mathcal J}$, we can iterate the process only a finite number of times and which concludes our proof.
\end{proof}

\begin{cor}\label{corr1} 
For $c\in (0,m_{\mathcal J})$, any $(PS)_c$ sequence of ${\mathcal E}_\lambda \! \mid_{{\mathcal N}_\lambda} $ is relatively compact. 
\end{cor}
\begin{proof}
Assume $(u_n)$ is a $(PS)_c$ sequence of ${\mathcal E}_\lambda \! \mid_{{\mathcal N}_\lambda}$. From Lemma \ref{compact} we have ${\mathcal J}(u_j)\geq m_{\mathcal J}$ and hence, it follows that up to a subsequence $u_n\to u$ strongly in $X_{V}^{s}(\R^{N})$ and $u$ is a solution of \eqref{fr}. 
\end{proof}
In order to finish the proof of Theorem \ref{gstate} we need the following result.
\begin{lemma}\label{flg}
$$
m_{\lambda}<m_{\mathcal J}.
$$
\end{lemma}
\begin{proof}
Let $Q\in X_{V}^{s}(\R^{N})$ be a groundstate solution of \eqref{squ}, we know that such a groundstate exists and for that we refer the reader to \cite{DSS2015}. Denote by $tQ$ the projection of $Q$ on ${\mathcal N_\lambda}$, that is, 
$t=t(Q)>0$ is the unique real number such that $tQ\in {\mathcal N_\lambda}$. Set
$$
A(Q)=\int_{\R^N} (I_\alpha*|Q|^p)|Q|^p\,,\quad B(Q)=\lambda \int_{\R^N} (I_\beta*|Q|^p)|Q|^p.
$$
As $Q\in {\mathcal N}_{\mathcal J}$ and $tQ\in {\mathcal N_\lambda}$, we get
\begin{equation}\label{g1}
||Q||^2=A(Q)
\end{equation}
and
$$
t^2\|Q\|^2=t^{2p}A(Q)+t^{2q}B(Q).
$$
From the above equalities, one can easily deduce that $t<1$. Therefore, we have

\begin{equation*}
\begin{aligned}
m_\lambda &\leq {\mathcal E}_\lambda(tQ)=\frac{1}{2}t^{2}\|Q\|^{2}-\frac{1}{2p}t^{2p}A(Q)-\frac{1}{2q}t^{2q}B(Q)\\
&= \Big(\frac{t^{2}}{2}-\frac{t^{2p}}{2p}\Big)\|Q\|^{2}-\frac{1}{2q}\Big(t^2||Q||^2-t^{2p}A(Q)\Big)\\
&= t^{2} \Big(\frac{1}{2}-\frac{1}{2q}\Big)\|Q\|^{2}+t^{2p}\Big(\frac{1}{2q}-\frac{1}{2p}\Big)\|Q\|^{2}\\
&< \Big(\frac{1}{2}-\frac{1}{2q}\Big)\|Q\|^{2}+\Big(\frac{1}{2q}-\frac{1}{2p}\Big)\|Q\|^{2}\\
&< \Big(\frac{1}{2}-\frac{1}{2p}\Big)\|Q\|^{2} ={\mathcal J}(Q)= m_{\mathcal J}.
\end{aligned}
\end{equation*}

\end{proof}

Further, using Ekeland Variational Principle, for any $n\geq 1$ there exists $(u_n) \in {\mathcal N}_\lambda$ such that
\begin{equation*}
\begin{aligned}
{\mathcal E}_\lambda(u_n)&\leq m_\lambda+\frac{1}{n} &&\quad\mbox{ for all } n\geq 1,\\
{\mathcal E}_\lambda(u_n)&\leq {\mathcal E}_\lambda(v)+\frac{1}{n}\|v-u_n\| &&\quad\mbox{ for all } v\in {\mathcal N}_\lambda \;\;,n\geq 1.
\end{aligned}
\end{equation*}
Now, one can easily deduce that $(u_n) \in {\mathcal N}_\lambda$ is a $(PS)_{m_\lambda}$ sequence for ${\mathcal E}_\lambda$ on ${\mathcal N}_\lambda$. Further, using Lemma \ref{flg} and Corollary \ref{corr1} we obtain that up to a subsequence $(u_n)$ converges strongly to some $u \in X_{V}^{s}(\R^{N})$ which is a groundstate of ${\mathcal E}_\lambda$.

\section{Proof of Theorem \ref{signc}}

In this section, we discuss the existence of a least energy sign-changing solution of \eqref{fr}. 

\subsection{Proof of Theorem }

\begin{lemma}\label{frl1}
Assume $p>q>2$ and $\lambda \in \R$. Then for any $u \in  X_{V}^{s}(\R^{N})$ and $u^{\pm} \neq 0$ there exists a unique pair $(\tau_0, \theta_0)\in (0, \infty)\times (0, \infty)$ such that $\tau_0 u^{+}+\theta_0 u^{-} \in {\cal M}_\lambda$. Furthermore, if $u\in {\cal M}_\lambda$ then for all $\tau$, $\theta\geq 0$ we have ${\mathcal E}_\lambda(u)\geq {\mathcal E}_\lambda(\tau u^{+}+\theta u^{-})$.
\end{lemma}

\begin{proof} We shall follow an idea developed in \cite{VX2017}. Denote
\begin{equation*}
\begin{aligned}
a_1&= ||u^{+}||_{X_{V}^{s}}^{2},&&\quad\mbox{  } b_1= ||u^{-}||_{X_{V}^{s}}^{2},\\
a_2&= \int_{\R^{N}}(I_{\alpha}*|u^{+}|^{p})|u^{+}|^{p},&&\quad\mbox{  } b_2= \int_{\R^{N}}(I_{\beta}*|u^{+}|^{q})|u^{+}|^{q}, \\
a_3&= \int_{\R^{N}}(I_{\alpha}*|u^{-}|^{p})|u^{-}|^{p},&&\quad\mbox{  } b_3= \int_{\R^{N}}(I_{\beta}*|u^{-}|^{q})|u^{-}|^{q}, \\
a_4&= \int_{\R^{N}}(I_{\alpha}*|u^{+}|^{p})|u^{-}|^{p},&&\quad\mbox{  } b_4= \int_{\R^{N}}(I_{\beta}*|u^{+}|^{q})|u^{-}|^{q},\\
A&= \int_{{\mathbb R}^N}\!\!\!\int_{{\mathbb R}^N}\frac{u^{+}(x)u^{-}(y)+u^{-}(x)u^{+}(y)}{|x-y|^{N+2s}}dxdy.
\end{aligned}
\end{equation*}

Let us define the function $ \Phi: [0, \infty)\times [0, \infty)\rightarrow \R$ by
\begin{equation*}
\begin{aligned}
\Phi(\tau, \theta)&= {\mathcal E}_\lambda(\tau^{\frac{1}{2p}} u^{+}+\theta^{\frac{1}{2p}} u^{-})\\
&= \frac{\tau^{\frac{1}{p}}}{2}a_1+\frac{\theta^{\frac{1}{p}}}{2}b_1-\lambda\frac{\tau^{\frac{q}{p}}}{2q}b_2 -\lambda\frac{\theta^{\frac{q}{p}}}{2q}b_3-\lambda\frac{\tau^{\frac{q}{2p}}\theta^{\frac{q}{2p}}}{2q}b_4-\frac{\tau}{2p}a_2 -\frac{\theta}{2p}a_3-\frac{\tau^{\frac{1}{2}}\theta^{\frac{1}{2}}}{2p}a_4-\tau^{\frac{1}{2p}}\theta^{\frac{1}{2p}}A.
\end{aligned}
\end{equation*}
Note that $\Phi$ is strictly concave. Therefore $\Phi$ has at most one maximum point. Also
\begin{equation}\label{fr2}
\lim_{\tau \rightarrow \infty}\Phi(\tau, \theta)= -\infty \mbox{ for all }\theta \geq 0 \quad\mbox{ and } \quad\mbox{ } \lim_{\theta \rightarrow \infty}\Phi(\tau, \theta)= -\infty \mbox{ for all }\tau \geq 0,
\end{equation}
and it is easy to check that 
\begin{equation}\label{fr3}
\lim_{\tau \searrow 0}\frac{\partial{\Phi}}{\partial{\tau}}(\tau, \theta)= \infty \mbox{ for all }\theta> 0 \quad\mbox{ and }  \lim_{\theta \searrow 0}\frac{\partial{\Phi}}{\partial{\theta}}(\tau, \theta)= \infty \mbox{ for all }\tau> 0.
\end{equation}
Hence, \eqref{fr2} and \eqref{fr3} rule out the possibility of achieving a maximum at the boundary. Therefore $\Phi$ has exactly one maximum point $(\tau_0, \theta_0)\in (0, \infty)\times (0, \infty)$.
\end{proof}

\begin{lemma}\label{frl2}
The energy level $c_\lambda>0$ is achieved by some $v\in {\cal M}_{\lambda}$.
\end{lemma}
\begin{proof}
Let $(u_n)\subset {\cal M}_{\lambda}$ be a minimizing sequence for $c_\lambda$. Note that
\begin{equation*}
\begin{aligned}
{\mathcal E}_\lambda(u_{n})&= {\mathcal E}_\lambda(u_{n})-\frac{1}{2q}\langle {\mathcal E}_\lambda'(u_{n}), u_{n}\rangle \\
&= \Big(\frac{1}{2}-\frac{1}{2q}\Big)\|u_{n}\|_{X_{V}^{s}}^{2}+\Big(\frac{1}{2q}-\frac{1}{2p}\Big)\int_{\R^{N}}(I_{\alpha}*|u|^{q})|u|^{q}\\
&\geq \Big(\frac{1}{2}-\frac{1}{2q}\Big)\|u_{n}\|_{X_{V}^{s}}^{2}\\
&\geq C_{1}\|u_{n}\|_{X_{V}^{s}}^{2},
\end{aligned}
\end{equation*}
where $C_{1}>0$ is a positive constant. Therefore, for some constant $C_{2}> 0$ we have 
$$
\|u_{n}\|_{X_{V}^{s}}^{2}\leq C_{2}{\mathcal E}_{\lambda}(u_{n})\leq M,
$$
which implies $(u_n)$ is bounded in $X^{s}_{V}(\R^{N})$. So, $(u_{n}^{+})$ and $(u_{n}^{-})$ are also bounded in $X^{s}_{V}(\R^{N})$ and passing to a subsequence, there exists $u^{+}$, $u^{-}\in H^{s}(\R^{N})$ such that
$$
u_{n}^{+}\rightharpoonup u^{+} \mbox{ and } u_{n}^{-}\rightharpoonup u^{-} \quad\mbox{ weakly in } X^{s}_{V}(\R^{N}).
$$
Since $p$, $q>2$ satisfy \eqref{p} and \eqref{q} we deduce that the embeddings $X^{s}_{V}(\R^{N})\hookrightarrow L^{\frac{2Np}{N+\alpha}}(\R^{N})$ and  $X^{s}_{V}(\R^{N})\hookrightarrow L^{\frac{2Nq}{N+\beta}}(\R^{N})$ are compact. Thus,
\begin{equation}\label{m1}
u_{n}^{\pm} \rightarrow u^{\pm} \quad\mbox{ strongly in } L^{\frac{2Np}{N+\alpha}}(\R^{N}) \cap L^{\frac{2Nq}{N+\beta}}(\R^{N}).
\end{equation}
Moreover, by Hardy-Littlewood-Sobolev inequality we estimate
\begin{equation*}
\begin{aligned}
C\Big(\|u_{n}^{\pm}\|_{L^{\frac{2Np}{N+\alpha}}}^{2}+\|u_{n}^{\pm}\|_{L^{\frac{2Nq}{N+\beta}}}^{2}\Big)\leq \|u_{n}^{\pm}\|_{X_{V}^{s}}^{2}&= \int_{\R^{N}}(I_{\alpha}*|u_n|^{p})|u_{n}^{\pm}|^{p}+|\lambda|\int_{\R^{N}}(I_{\beta}*|u_n|^{q})|u_{n}^{\pm}|^{q}\\
&\leq
C\Big(\|u_{n}^{\pm}\|_{L^{\frac{2Np}{N+\alpha}}}^{p}+\|u_{n}^{\pm}\|_{L^{\frac{2Np}{N+\alpha}}}^{p}\Big)\\
&\leq
C\Big(\|u_{n}^{\pm}\|_{L^{\frac{2Np}{N+\alpha}}}^{2}+\|u_{n}^{\pm}\|_{L^{\frac{2Nq}{N+\beta}}}^{2}\Big)\Big(\|u_{n}^{\pm}\|_{L^{\frac{2Np}{N+\alpha}}}^{p-2}+||u_{n}^{\pm}||_{L^{\frac{2Nq}{N+\beta}}}^{q-2}\Big).
\end{aligned}
\end{equation*}
Since $u_{n}^{\pm}\neq 0$, we deduce 
\begin{equation}\label{m2}
\|u_{n}^{\pm}\|_{L^{\frac{2Np}{N+\alpha}}}^{p-2}+\|u_{n}^{\pm}\|_{L^{\frac{2Nq}{N+\beta}}}^{q-2}\geq C> 0 \quad\mbox{ for all } n\geq 1.
\end{equation}
Hence, by \eqref{m1} and \eqref{m2} it follows that $u^{\pm} \neq 0$.
Further using \eqref{m1} and Hardy-Littlewood-Sobolev inequality, we have
\begin{equation*}
\begin{aligned}
&\int_{\R^{N}}(I_{\alpha}*|u_{n}^{\pm}|^{p})|u_{n}^{\pm}|^{p} &&\rightarrow \int_{\R^{N}}(I_{\alpha}*|u^{\pm}|^{p})|u^{\pm}|^{p},\\
&\int_{\R^{N}}(I_{\alpha}*|u_{n}^{+}|^{p})|u_{n}^{-}|^{p} &&\rightarrow \int_{\R^{N}}(I_{\alpha}*|u^{+}|^{p})|u^{-}|^{p},\\
&\int_{\R^{N}}(I_{\beta}*|u_{n}^{\pm}|^{q})|u_{n}^{\pm}|^{q} &&\rightarrow \int_{\R^{N}}(I_{\beta}*|u^{\pm}|^{q})|u^{\pm}|^{q},
\end{aligned}
\end{equation*}
and
\begin{equation*}
\int_{\R^{N}}(I_{\beta}*|u_{n}^{+}|^{q})|u_{n}^{-}|^{q} \rightarrow \int_{\R^{N}}(I_{\beta}*|u^{+}|^{q})|u^{-}|^{q}.
\end{equation*}
By Lemma \ref{frl1}, there exists a unique pair $(\tau_{0}, \theta_{0})$ such that $\tau_{0} u^{+}+\theta_{0} u^{-}\in {\cal M}_{\lambda}$. By the weakly lower semi-continuity of the norm $\|.\|_{X^{s}_V}$, we deduce that
\begin{equation*}
\begin{aligned}
c_\lambda \leq {\mathcal E}_\lambda(\tau_{0} u^{+}+\theta_{0} u^{-})&\leq \liminf_{n\rightarrow \infty} {\mathcal E}_\lambda(\tau_{0} u^{+}+\theta_{0} u^{-})\\
&\leq \limsup_{n\rightarrow \infty} {\mathcal E}_\lambda(\tau_{0} u^{+}+\theta_{0} u^{-})\\
&\leq \lim_{n\rightarrow \infty}{\mathcal E}_\lambda(u_{n})\\
&= c_\lambda.
\end{aligned}
\end{equation*}
Letting now $v= \tau_{0} u^{+}+\theta_{0} u^{-}\in {\cal M}_\lambda$, we finish the proof.
\end{proof}

\begin{lemma}\label{frl3}
$v= \tau_{0} u^{+}+\theta_{0} u^{-}\in {\cal M}_\lambda$ is a critical point of ${\mathcal E}_\lambda:X^{s}_{V}(\R^{N}) \rightarrow \R$, that is 
$$
{\mathcal E}_\lambda'(v)=0.
$$
\end{lemma}
\begin{proof}
Assume by contradiction that $v$ is not a critical point of ${\mathcal E}_{\lambda}$. Then there exists $\varphi\in C_{c}^{\infty}(\R^N)$ such that
$$
\langle {\mathcal E}_\lambda'(v), \varphi \rangle= -2.
$$
Since ${\mathcal E}_{\lambda}$ is continuous and differentiable, there exists $r>0$ small such that
\begin{equation}\label{fr4}
\langle {\mathcal E}_\lambda'(\tau u^{+}+\theta u^{-}+\varepsilon \bar{v}), \bar{v} \rangle \;\; \leq -1 \quad\mbox{ if } (\tau- \tau_{0})^{2}+(\theta- \theta_0)^{2}\leq r^{2} \mbox{ and } 0\leq \varepsilon \leq r.
\end{equation}
Let $D$ be the open disc in $\R^{2}$ of radius $r>0$ centered at $(\tau_0, \theta_0)$. We define a continuous function $\psi: D\rightarrow [0, 1]$ as 
\begin{equation*}
\psi(\tau, \theta)= \left\{\begin{array}{cc}1\quad\mbox{ if }(\tau- \tau_{0})^{2}+(\theta- \theta_0)^{2}\leq \frac{r^{2}}{16}, \\ 0\quad\mbox{ if }(\tau- \tau_{0})^{2}+(\theta- \theta_0)^{2}\geq \frac{r^{2}}{4}.\end{array}
\right.
\end{equation*}
Further, we define a continuous map $S: D\rightarrow X^{s}_{V}(\R^{N})$ as 
\begin{equation*}
S(\tau, \theta)= \tau u^{+}+\theta u^{-}+r\psi(\tau, \theta)\bar{v} \quad\mbox{ for all } (\tau, \theta)\in D
\end{equation*}
and $L: D\rightarrow \R^{2}$ as
\begin{equation*}
L(\tau, \theta)= (\langle {\mathcal E}_\lambda'(S(\tau, \theta)), S(\tau, \theta)^{+}\rangle, \langle {\mathcal E}_\lambda'(S(\tau, \theta)), S(\tau, \theta)^{-}\rangle) \quad\mbox{ for all }(\tau, \theta)\in D. 
\end{equation*}
Since the mapping $u \mapsto u^{+}$ is continuous in $X^{s}_{V}(\R^{N})$, it follows that $L$ is continuous. If $(\tau- \tau_{0})^{2}+(\theta- \theta_0)^{2}= r^{2}$, that is, if we are on the boundary of $D$ then $\psi= 0$ by definition. Then $S(\tau, \theta)= \tau u^{+}+\theta u^{-}$ and using Lemma \ref{frl1}, we get  
\begin{equation*}
L(\tau, \theta)= (\langle {\mathcal E}_\lambda'(\tau u^{+}+\theta u^{-}), (\tau u^{+}+\theta u^{-})^{+}\rangle, \langle {\mathcal E}_\lambda'(\tau u^{+}+\theta u^{-}), (\tau u^{+}+\theta u^{-})^{-} \rangle)\neq 0 \quad\mbox{ on } {\partial D}. 
\end{equation*}
Therefore, the Brouwer degree is well-defined and $\deg(L, {\rm int} (D), (0, 0))=1$. Then, there exists $(\tau_1, \theta_1)\in {\rm int} (D)$ such that $L(\tau_1, \theta_1)= (0, 0)$. Thus, $S(\tau_1, \theta_1)\in {\cal M}_{\lambda}$ and using the definition of $c_\lambda$ we get
\begin{equation}\label{fr5}
{\mathcal E}_\lambda(S(\tau_1, \theta_1))\geq c_\lambda.
\end{equation}
Using equation \eqref{fr4}, we deduce that
\begin{equation}\label{fr6}
\begin{aligned}
{\mathcal E}_\lambda(S(\tau_1, \theta_1))&= {\mathcal E}_\lambda(\tau_1 u^{+}+\theta_{1} u^{-})+\int_{0}^{1}\frac{d}{dt}{\mathcal E}_\lambda(\tau_1 u^{+}+\theta_{1} u^{-}+rt \psi(\tau_1, \theta_1)\bar{v})dt \\
&=  {\mathcal E}_\lambda(\tau_1 u^{+}+\theta_{1} u^{-})+\int_{0}^{1}\langle {\mathcal E}_\lambda'(\tau_1 u^{+}+\theta_{1} u^{-}+rt \psi(\tau_1, \theta_1)\bar{v}),r\psi(\tau_1, \theta_1)\bar{v} \rangle dt \\
&= {\mathcal E}_\lambda(\tau_1 u^{+}+\theta_{1} u^{-})-r\psi(\tau_1, \theta_1).
\end{aligned}
\end{equation}
If $(\tau_1, \theta_1)=(\tau_0, \theta_0)$, then $\psi(\tau_1, \theta_1)=1$ by definition and we deduce that
$$
{\mathcal E}_\lambda(S(\tau_1, \theta_1))\leq {\mathcal E}_\lambda(\tau_1 u^{+}+\theta_{1} u^{-})-r\leq c_\lambda-r< c_\lambda,
$$
and if  $(\tau_1, \theta_1)\neq (\tau_0, \theta_0)$, then using Lemma \ref{frl1} we have 
$$
{\mathcal E}_\lambda(\tau_1 u^{+}+\theta_{1} u^{-})< {\mathcal E}_\lambda(\tau_0 u^{+}+\theta_{0} u^{-})= c_\lambda,
$$
which yields
$$
{\mathcal E}_\lambda(S(\tau_1, \theta_1))\leq {\mathcal E}_\lambda(\tau_1 u^{+}+\theta_{1} u^{-})< c_\lambda,
$$
which is a contradiction to equation \eqref{fr5}. 
\end{proof}

\end{document}